\documentclass[a4paper,10pt]{article}
\usepackage[utf8]{inputenc}
\usepackage[english]{babel}
\usepackage{amsmath}
\usepackage{amssymb}
\usepackage{graphicx}
\usepackage{amsthm}
\usepackage{comment}
\newtheorem{theorem}{Theorem}[section]
\newtheorem{corollary}{Corollary}[theorem]
\newtheorem{proposition}[theorem]{Proposition}
\newtheorem{lemma}[theorem]{Lemma}
\newtheorem{claim}[theorem]{Claim} 
\title{Minimal surfaces of finite total curvature in $\mathbb{M}^2 \times \mathbb{R}$}
\author{Rafael Ponte}
\date{}
\begin{document}
\newcommand{\Addresses}{{% additional braces for segregating \footnotesize
  \bigskip
  \footnotesize

  \textsc{IMPA, Estrada Dona Castorina, 110, 22460-320, Rio de Janeiro,
Brazil}\par\nopagebreak
  \textit{E-mail address}: \texttt{prafael@impa.br}}}
\maketitle

\begin{abstract}
The goal of this article is to study minimal surfaces in $\mathbb{M}^2 \times \mathbb{R}$ having finite total curvature, where $\mathbb{M}^2$ is a Hadamard manifold. The main 
result gives a formula to compute the total curvature in terms of topological, geometrical and conformal data of the minimal surface. In particular, we prove the total 
curvature is an integral multiple of $2\pi$.
\end{abstract}

\section{Introduction}
Minimal surfaces with finite total curvature in three-dimensional spaces have been widely studied in the recent decades. A classical result in this subject states that,
if $\Sigma \subset \mathbb{R}^3$ is a complete immersed minimal surface of finite total curvature, then $\Sigma$ is conformally equivalent to a compact Riemann surface with a 
finite number of points removed. Moreover, its Weierstrass data can be extended meromorphically to the punctures and its total curvature is an integral 
multiple of $4 \pi$ (see [O] for those results). Other references for finite total curvature minimal surfaces in $\mathbb{R}^3$ are [F], [HK] and [W].
In [JM], the authors obtain a formula for the total curvature of a minimal surface $\Sigma$ in terms of topological and geometrical invariants (see also [F] for a discussion of 
these results). 

In [HR], minimal surfaces of finite total curvature in $\mathbb{H}^2 \times \mathbb{R}$ are studied. Analogues of the above results were proved there; mainly, a Jorge-Meeks-type
formula.

In this work, we generalize [HR] to the case of $\mathbb{M} \times \mathbb{R}$; $\mathbb{M}$ a Hadamard surface such that the sectional curvature satisfies the inequalities 
$-a^2 \leq K_{\mathbb{M}} \leq -b^2$, where $a$ and $b$ are positive constants. Inspired by [PR], we add a refinement to the generalization. We also present some examples of 
minimal surfaces with finite total curvature in $\mathbb{M} \times \mathbb{R}$.

Here, the main tools are the comparison theorems, which allow us to construct complete mean convex barriers. It is expected that these comparison theorems will be used to 
construct examples of noncompact minimal surfaces in $\mathbb{M} \times \mathbb{R}$.

The outline of the paper is the following: in Section \ref{prelims}, we establish the fundamental concepts related to the problem. In Section \ref{mainsection}, we prove the main
results. In Section \ref{example}, we describe some examples of minimal surfaces with finite total curvature and, in Section \ref{appendix}, we discuss some useful facts that are
crucial in some proofs.

This work is part of the author's doctoral thesis at Instituto Nacional de Matemática Pura e Aplicada (IMPA). The author would like to thank his advisor Harold Rosenberg for 
the patience, support and guidance. 

\section{Preliminaries}\label{prelims}
Let $X : \Sigma \rightarrow \mathbb{M} \times \mathbb{R}$ be a minimal conformal immersion of the surface $\Sigma$ in $\mathbb{M} \times \mathbb{R}$, where $\mathbb{M}$ is 
a Hadamard surface satisfying 
$-a^2 \leq K_{\mathbb{M}} \leq -b^2$, for positive constants $a$ and $b$. We can decompose the 
immersion $X$ as $(h , f)$, where $h$ and $f$ are the projections of $X$ in the first and second factors of $\mathbb{M} \times \mathbb{R}$, respectively. Since $X$ is minimal, 
the maps $h$ and $f$ are harmonic. 

We consider local conformal parameters for a simply-connected open domain $\Omega \subset \Sigma$, given by $w = u + iv$. The metric of $\Sigma$ in these
parameters has the expression ${\lambda}^2 |dw|^2$. In $\mathbb{M}$, we can take global conformal parameters $z = x + iy$ such that $\mathbb{M}$ is isometric to 
$(\mathbb{D} , \frac{4 \alpha(z)^2}{(1 - |z|^2)^2}|dz|^2)$; $\alpha$ a smooth function bounded between two positive constants (see [LTW]). 
%since this surface is conformally equivalent to the 
%Poincaré disc, being the conformal factor bounded by two positive constants (see [LTW]). 
With these notations, we can write the equation satisfied by the harmonic map $h$: 
\begin{center}
 $\sigma h_{w \bar{w}} + 2 (\sigma_z \circ h) h_w h_{\bar{w}} = 0$,
\end{center}
where $\sigma(z)^2 =\frac{4 \alpha(z)^2}{(1 - |z|^2)^2}$.

Associated to this map, we have the holomorphic \textit{Hopf differential} of $h$, given by 
\begin{center}
$\mathcal{Q}(h) = (\sigma \circ h)^2 h_w \bar{h}_{w} dw^2$  
\end{center}
(for short, we write $\phi$ for $(\sigma \circ h)^2 h_w \bar{h}_{w}$). 
\begin{comment}
Another object is the 
\textit{complex coefficient of dilatation} of a quasiconformal map:
\begin{center}
 $a(h) = \frac{\displaystyle \overline{h_{\bar{w}}} }{\displaystyle h_w}$.
\end{center}
\end{comment}

Since $X$ is a conformal immersion, the following equalities hold:
\begin{center}
$\sigma^2 |h_u|^2 + f_u^2 = \sigma^2 |h_v|^2 + f_v^2$; 
 
$\sigma^2 \langle h_u , h_v \rangle + f_u f_v = 0$. 
\end{center}

A trivial consequence of the above equations is that $\phi = -f_w^2$, hence the zeroes of $\phi$ have even order. Furthermore, we define $\eta$ as the holomorphic 1-form in 
$\Omega$ given by $\eta = - 2i \sqrt{\phi(w)}dw$, where the square root of $\phi$ is chosen in such a way that 
\begin{align}\label{detf}
f = Re \int_w \eta. 
\end{align}
Considering $N : \Omega \rightarrow \mathbb{S}^2$ the Gauss map, we can write $N$ as $N_1 \partial_x + N_2 \partial_y + N_3 \partial_t$, where $x+iy$ is a local conformal 
parameter of $\mathbb{M}$ and $t \in \mathbb{R}$. We have that there exists a map $g: \Omega \rightarrow \mathbb{C}$ such that 
\begin{center}
 $N = (N_1 , N_2 , N_3)= \frac{\displaystyle ((2/\sigma)Re(g) , (2/\sigma)Im(g) , |g|^2 - 1)}{\displaystyle |g|^2 + 1}$, 
\end{center}
and it satisfies 
\begin{center}
$g^2 = -\frac{\displaystyle h_w}{\displaystyle \overline{h_{\bar{w}}}}$.
\end{center}

\begin{comment}We have that the 1-form $dh$ can be written as 
\begin{center}
$dh = h_w dw + h_{\bar{w}} d\bar{w} = (2 \sigma)^{-1}(\overline{g^{-1} \eta} - g \eta)$.
\end{center}
\end{comment}

The metric of $\Omega$ can also be written as 
\begin{center}
$ds^2 = \sigma^2 (|h_w| + |h_{\bar{w}}|)^2 |dw|^2 = \frac{1}{4}(|g|^{-1} + |g|)^2|\eta|^2$.
\end{center}

Define $\xi$ as the function given by $\xi : = log |g|$. It is known (see [SY]) that the following equation holds:
\begin{center}
 $\Delta_0 log \frac{|h_w|}{|h_{\bar{w}}|} = -2K_{\mathbb{M}} J(h)$,
\end{center}
where $J(h)$, the \textit{Jacobian} of $h$, denotes the expression $\sigma^2(|h_w|^2 - |h_{\bar{w}}|^2)$, and $\Delta_0$ stands for the Euclidean Laplacian. 

Based on the previous facts, we obtain that the function $\xi$ satisfies the sinh-Gordon equation: 
\begin{align}\label{gordon}
\Delta_0 \xi = -2K_{\mathbb{M}}sinh (2 \xi) |\phi|.
\end{align}
Rewriting some expressions in terms of $\xi$, we have: 
\begin{align*}
ds^2 & = 4cosh^2 (\xi) |\eta|^2 \\
N_3 & = tanh (\xi).
\end{align*}
Finally, we denote by $K_{\Sigma}$ the Gaussian curvature of $\Sigma$. The Gauss equation states that
\begin{center}
$K_{\Sigma} = K_{\mathbb{M} \times \mathbb{R}} (X_u , X_v)+ K_{ext}$, 
\end{center}
where $K_{ext}$ is the extrinsic curvature of $\Sigma$. Since $X$ is minimal and the sectional curvature of $\mathbb{M} \times \mathbb{R}$ is nonpositive, the curvature 
$K_{\Sigma}$ is nonpositive. The total curvature of $\Sigma$ is defined by 
\begin{center}
 $C(\Sigma) = \int_{\Sigma} K_{\Sigma} dA$.
\end{center}
\begin{comment}\begin{proposition} Let $h : \Omega \rightarrow \mathbb{M}$ be a harmonic mapping from a simply connected domain $\Omega \subset \mathbb{C}$ such that the holomorphic quadratic 
differential $\mathcal{Q}(h)$ does not vanish or has zeros with even order. Then there exists a complex map $g$ and a holomorphic one-form $\eta = \pm 2 i \sqrt{\mathcal{Q}(h)}$ 
such that, setting $f = Re \int_w \eta$, the map $X := (h, f ) : \Omega \rightarrow \mathbb{M} \times \mathbb{R}$ is a conformal and minimal (possibly branched) immersion. If 
$\xi := ln |g|$, the third component of the unit normal vector is given by $N_3 = tanh \xi$ and the metric of the immersion is given by $ds^2 = cosh^2 \xi |\eta|^2$.
Moreover, $\xi$ is a solution of the sinh-Gordon equation
\begin{center}
$\Delta_0 \xi = -2K_{\mathbb{M}} sinh(2\xi)|\phi|$.
\end{center}
\end{proposition}
\end{comment}
\section{Minimal surfaces of finite total curvature}\label{mainsection}
We are going to prove the following result:

\begin{theorem}\label{main} Let $X$ be a complete minimal immersion of $\Sigma$ in $\mathbb{M} \times \mathbb{R}$ with finite total curvature. Then
\begin{enumerate}
  \item $\mathcal{Q}$ is holomorphic on $S \backslash \{ p_1, \cdots , p_n \}$ and extends meromorphically to each puncture. Moreover, parametrizing a neighborhood of each 
  puncture $p_j$ by the exterior of a disc of radius $R_j$ and writing 
  \begin{center}
  $\mathcal{Q}(z) = (\sum_{k \geq 1} \frac{a_{-k}}{z^k} + P_j(z))^2 dz^2$ 
  \end{center}
  around $p_j$, where $P_j$ is a polynomial function, then $P_j$ is not identically zero. We denote the degree of $P_j$ by $m_j$.
\item The third coordinate of the unit normal vector $N_3$ converges to $0$ uniformly at each puncture.
\item The total curvature is a multiple of $2\pi$. More precisely, the following equality holds:
\begin{center}
 $\displaystyle \int_{\Sigma} K_{\Sigma} = 2 \pi \displaystyle (2 - 2g - 2n - \sum_{k=1}^{n} m_k)$.
\end{center}
\end{enumerate}
\end{theorem}

\begin{proof}
It is well known that $\Sigma$ is conformally equivalent to $S \backslash \{ p_1, \cdots , p_n \}$, a compact Riemann surface $S$ punctured in a finite number of points. This 
 follows directly from Huber's theorem (see [Hu]).
\begin{enumerate}
 \item For $j = 1, \dots , n$, let $U_j$ be a neighborhood of $p_j$ such that $U_j \cap U_k = \emptyset$ if $j \neq k$ and there exists a biholomorphism 
 $\psi_j : D(0 , 1) \rightarrow U_j$ satisfying $\psi_j (0) = p_j$. 
 If $0 < r < 1$, define $U_j (r)$ by $\psi_j (D(0 , r))$, the set $S(r)$ by $S \backslash \cup^{n}_{k=1} U_k (r)$ and $S^{*}$ by $S \backslash \{ p_1 , \cdots , p_n \}$. Around 
 $p_j$, we can take $U_j(r) \backslash \{ p_j \}$ as a neighborhood of this puncture in $S^{*}$, and the corresponding end representative of $\Sigma$ can be parametrized by 
 $A(1/r) := \mathbb{C} \backslash \overline{D(0 , r)}$. %(from now on, we are going to consider this coordinate domain, unless otherwise stated). 
 In these coordinates, the metric is given by $ds^2 = {\lambda}^2 |dz|^2 = 4 cosh^2 (\xi) |\phi| |dz|^2$.
 
 If $u := log cosh^2 (\xi)$, we have that 
 \begin{center}
  $\Delta_0 u = \frac{2||\nabla_0 \xi||^2}{cosh^2(\xi)} + 2 tanh (\xi) \Delta_0 \xi = \frac{2||\nabla_0 \xi||^2}{cosh^2(\xi)} 
  - 8K_{\mathbb{M}}sinh^2 (\xi) |\phi| \geq 0$.
 \end{center}

 Clearly, $u$ is a subharmonic function. %Moreover, it is globally defined on $\Sigma$, since it is true for $g^2$ and $u := ln cosh^2 (log(|g|))$. 
 
 \begin{claim} \label{claim1} The quadratic differential $\mathcal{Q}$ has a finite number of zeroes in $S$.
 \end{claim}
\begin{proof}
Clearly, the number of zeroes in $S(r)$ is finite, since they are isolated and $S(r)$ is compact. Fixing $j \in \{ 1, \dots , n\}$, consider an end 
representative $E_j$ corresponding to $p_j$ and its parametrization set given by $A(R_j)$. 
If $\mathcal{Z}_j$ is the set of zeroes of $\phi$ in $E_j$, we have that
\begin{center}
 $\Delta_0 log |\phi| = \sum_{z \in \mathcal{Z}_j} 2\pi m_z \delta_{z}$,
\end{center}
where $m_z$ is the multiplicity of $z$ as a zero of $\phi$.

It is well-known that $-K_{\Sigma} {\lambda}^2 = \Delta_0 log {\lambda}$, hence $\Delta_0 u = - \Delta_0 log |\phi| - 2 K_{\Sigma} {\lambda}^2$. In the annulus 
$C(R) = \{ R_j \leq |z| \leq R\}$, we have that
\begin{center}
 $- 2\int_{C(R)} K_{\Sigma} dA - \sum_{z \in \mathcal{Z}_j \cap C(R)} 2\pi m_z = \int_{C(R)} \Delta_0 u \geq 0$, 
\end{center}
then $\sum_{z \in \mathcal{Z}_j \cap C(R)} m_z$ is uniformly bounded on $(R_j , \infty)$, and Claim \ref{claim1} is proved.
\end{proof}

A trivial corollary of last claim is that $\int_{C(R)} \Delta_0 u$ is nonnegative and bounded from above by $2C(\Sigma)$, consequently the integral $\int_{C(R)} \Delta_0 u$ is 
uniformly bounded on $(R_j , \infty)$.

\begin{claim} The inequality $cosh^2(\xi) |\phi| \leq B |z|^B |\phi|$ holds in $A(R)$, for a positive constant $B > 0$ and for 
sufficiently large $R>0$.\end{claim}
\begin{proof} This follows the same ideas of the analogous result in [HR].
\end{proof}
\begin{claim} The differential $\mathcal{Q}$ is holomorphic on $S \backslash \{ p_1, \cdots , p_n \}$ and extends meromorphically to each puncture.\end{claim}
\begin{proof} Considering $R$ to be large enough, we can take $B$ as an even integer and $\phi$ as a function without zeroes in $A(R)$. In 
that case, we have that $B z^B \phi$ is the square of a holomorphic function $\rho$. We obtain that $cosh(\xi) |\phi|^{\frac{1}{2}} \leq |\rho|$, and by Lemma 9.6 of [O],
the function $\rho$ extends meromorphically to infinity, hence we can extend $\phi$ meromorphically to the puncture.
\end{proof}
\begin{comment}
If $I(R) := \int_{0}^{2 \pi} u(R, \theta) d\theta$, then, for large $R$,
\begin{center}
 $I(R) \leq A ln(R) + B$,
\end{center}
for positive constants $A, B > 0$.

\bfseries Claim 3. \mdseries For each puncture of $\Sigma$, we have that $cosh^2(\xi) |\phi| \leq \beta |z|^\alpha |\phi|$. 

By the subharmonicity, $u(z) \leq \alpha ln |z| + \beta$. Then $2ln \lambda = u + ln |\phi| \leq \alpha ln |z| + \beta + ln |\phi|$, which gives us the claim. 
Then, we can extend $\phi$ meromorphically to the puncture, by Lemma 9.6 of [O].
\end{comment}
\begin{claim} If the differential $\mathcal{Q}$ is written as 
\begin{center}
$\mathcal{Q}(z) = (\sum_{k \geq 1} \frac{a_{-k}}{z^k} + P_j(z))^2 dz^2$  
\end{center}
around $p_j$, where $P_j$ is a polynomial function, then $P_j$ is not identically zero.\end{claim}

\begin{proof} First, we are going to prove the claim when $a_{-1} = 0$. In fact, if the claim is false in this case, then, up to a conformal change of coordinates, we can suppose
that $\mathcal{Q}(z) = z^{2k_j} dz^2$, for some integer $k_j$ satisfying $k_j \leq -2$. In this situation, the integral $\int_{A(R)} |\phi(z)| dz$ is finite. 
Therefore, denoting by $E'_j$ the associated end representative, we obtain that
\begin{align*}
& \int_{E'_j} -K_{\Sigma} dA = \int_{A(R)} \Delta_0 log \lambda dz = \int_{A(R)} \Delta_0 u dz \\
& \geq \int_{A(R)} \frac{2||\nabla_0 \xi||^2}{cosh^2(\xi)} dz - \int_{A(R)} 8K_{\mathbb{M}}sinh^2 (\xi) |\phi| dz \\
& \geq \int_{A(R)} 8 b^2 sinh^2 (\xi) |\phi| dz, 
\end{align*}

consequently the following inequality holds:
\begin{center} $\int_{A(R)} 8 b^2 |\phi| dz - \int_{E'_j} K_{\Sigma} dA \geq \int_{A(R)} 8 b^2 cosh^2 (\xi) |\phi| dz$.\end{center}
 We conclude that $Area(E'_j) = \int_{A(R)} 4 cosh^2 (\xi) |\phi| dz$ is finite, which contradicts the fact that a complete end of $\Sigma$ must have infinite area (see Remark 
 4 of [Fr]). 
 
 Now we prove the claim when $a_{-1}$ is nonzero. Indeed, suppose the end associated to $p_1$ satisfies $a_{-1} \neq 0$ and $P_1 \equiv 0$. Parametrizing an end representative of 
 $p_1$ (denote by $E_1$) by $A(R)$ such that $\mathcal{Q}(z) = -\beta^2 z^{-2}dz^2$, for some $\beta > 0$, we obtain the equality 
 \begin{center}
 $f(z) = 2 \beta Re(\int_z u^{-1} du) = 2 \beta log (|z|/R)$. 
 \end{center} 
 We conclude that the intersection of $E_1$ with $\mathbb{M} \times \{ t \}$ is a compact curve, for $t \geq 0$.
 
 Since $K_{\mathbb{M}} \leq -b^2$, we can take a vertical rotational catenoid $C$ in $\mathbb{M} \times \mathbb{R}$ whose mean curvature vector field points inwards, whose height is 
 bounded and such that $\partial E_1$ is disjoint from all vertical translations of $C$ (see the Appendix for the meaning of "inwards" and the existence of such catenoid). Then, 
 if $T_x(C)$ is a vertical translation of $C$ by $x \in \mathbb{R}$, we have $T_{-n}(C) \cap E_1 = \empty$ for large enough $n \in \mathbb{N}$. Moving the catenoid vertically in 
 the positive direction, since the catenoid can not have a first point of contact with $E_1$, by the maximum principle, we have that $E_1$ is cylindrically bounded, and it has 
 unbounded height. But this contradicts Proposition \ref{cylinder}, then $P_1$ must not be identically zero.
 \end{proof}
 \textbf{Remark.} Since the polynomials $P_j$ are not identically zero, we can parametrize an end representative of $p_j$ (denote by $E_j$) by 
 $A(R_j)$ such that, by Theorem 6.4 of [St], $\mathcal{Q}(z) = ((m_j + 1)z^{m_j} + \frac{ci}{z})^{2}dz^2$, for some $c \in \mathbb{R}$ (the 
 coefficient $c$ is real because the function $f$ is well defined by (\ref{detf})). We are going to assume this expression for the 
 Hopf differential of $h$ near $p_j$, unless otherwise stated.
 
 \textbf{Remark.} There are several manners to prove that $P_j$ is not the zero polynomial when $a_{-1} \neq 0$. Consider in $\mathbb{M}$ the Fermi coordinates given by 
 $\phi(s,t) = exp_{\alpha(t)} (s J\alpha'(t))$, for $(s,t) \in \mathbb{R}^2$ and some geodesic $\alpha$ which does not intersect $h(\partial E_1)$. In order to prove that $E_1$ 
 is cylindrically bounded, we could use the barriers defined by the graph of the function
 \begin{align*}
  f(s) = \frac{1}{k}log(tanh(\frac{ks}{2})), s>0,
 \end{align*}
where $k \in (0,b)$. Supposing that $h(\partial E_1)$ is contained in the region $\{ \phi(s,t) \in \mathbb{M}; s < 0 \}$, we have that the mean 
curvature vector field of the graph of $f$ points upwards (see [GR] for the proof), and proceeding as before, we conclude that $h(E_1)$ is contained in the convex hull of 
$h(\partial E_1)$, therefore $E_1 \subset D(p, R) \times \mathbb{R}$, for some $p \in \mathbb{M}$, $R>0$. In addition, we can prove that $E_1$ can not be cylindrically bounded 
considering a family of rotational catenoids with mean curvature vector field pointing inwards. We suppose this family varies from a surface containing $D(p, R) \times \mathbb{R}$ 
in its complement to a double-sheeted covering of a horizontal slice $\mathbb{H}^2 \times \{ t \}$, for a sufficiently large $t>0$ (the existence of this family of catenoids is 
guaranteed in the Appendix). Then, when we vary the catenoids, we obtain a first point of contact of $E_1$ and one of the annuli, a contradiction to the maximum principle 
(see [PR]). 
 
 \item We prove here that $N_3$ goes to $0$ in each puncture. Around a puncture $p_j$, we consider a parametrization of the associated end representative $E_j$ by 
 $A(R_j)$. We choose $R_j$ to guarantee that $\phi$ has no zeroes in $E_j$ and, in this situation, it is clear that the metric 
$g_{\phi} = |\phi(z)||dz|^2$ is flat. Denoting by $D_{|\phi|}(z , r)$ a disc in $E_j$ with respect to the metric $g_{\phi}$ centered in $z$ of radius $r$, by Proposition 2.1 
and Lemma 2.4 of [HNST] (which also can be applied to this context), there exist positive constants $R'_1$ and $c_1$ such that, if $|z| > R'_1$, then $F := \int \sqrt{\phi} dz$ 
is well defined in $D_{|\phi|}(z , c_1 |z|)$ and it is a conformal diffeomorphism onto its image. If $w$ are the coordinates in $D_{|\phi|}(z , c_1 |z|)$ induced by $F$ such that 
$w(z) := F(z) = 0$, we have that $g_{\phi} = |dw|^2$. Therefore, if $|z| > max\{ c_1 ^{-1} , R'_1\}$, define in $D_{|\phi|} (z , 1)$ the metric 
\begin{center}
$d\mu^2 = \sigma^2 |dw|^2 := \frac{\displaystyle 4 \alpha(w)^2}{\displaystyle (1 - d_{|\phi|}(w, 0)^2)^2}|dw|^2$,
\end{center}
where $d_{\phi}$ is the distance function in the metric $|dw|^2$. Notice that this metric is precisely the metric of $\mathbb{M}$ in the disc $D_{|\phi|} (z , 1)$. Then its 
curvature function, denoted by $\tilde{K}$, satisfies the inequalities $-a^2 \leq \tilde{K} \leq -b^2$.

The functions $\xi$ and $\tilde{\xi} := log \sigma$ satisfy 
\begin{align*}
\Delta_{|\phi|} \xi & = -2K_{\mathbb{M}}sinh (2 \xi);\\
\Delta_{|\phi|} \tilde{\xi} & = -\tilde{K} e^{2 \tilde{\xi}}. 
\end{align*}

If $\eta := \xi - \tilde{\xi}$, we have 
\begin{align*}
\Delta_{|\phi|} \eta & = -K_{\mathbb{M}}(e^{2 \xi} - e^{- 2 \xi} - \frac{\tilde{K}}{K_{\mathbb{M}}} e^{2 \tilde{\xi}}) \\
& \geq b^2 e^{2 \xi} - a^2 e^{- 2 \xi} - a^2 e^{2 \tilde{\xi}} \\
& \geq e^{2\tilde{\xi}}(b^2 e^{2 \eta} - a^2 C e^{- 2 \eta} - a^2), 
\end{align*}
where $C : = max_{w \in D_{|\phi|} (z , 1)} e^{-4\tilde{\xi}(w)}$. By the Cheng-Yau's maximum principle (see [CY] for the statement), $\eta$ is bounded from above 
and it has a maximum at a point $p_0 \in D_{|\phi|}(z , 1)$. Obviously, $\Delta_{\mu} \eta (p_0) \leq 0$, then, at this point, 
\begin{align*}
 -e^{2\tilde{\xi}}K_{\mathbb{M}}(e^{2 \eta} - e^{- 4 \tilde{\xi}} e^{- 2 \eta} - \frac{\tilde{K}}{K_{\mathbb{M}}}) \leq 0 & \leftrightarrow \\
 e^{2 \eta} - e^{- 4 \tilde{\xi}} e^{- 2 \eta} \leq \frac{\tilde{K}}{K_{\mathbb{M}}} \leq \frac{a^2}{b^2} & \leftrightarrow \\
 e^{4 \eta} -\frac{a^2}{b^2} e^{2 \eta} - e^{- 4 \tilde{\xi}} \leq 0 & \leftrightarrow \\
 2 e^{2 \eta(p_0)} \leq \frac{a^2}{b^2} + \sqrt{\frac{a^4}{b^4} + 4 C} =: 2C_1.
\end{align*}
Since $\eta$ maximizes at $p_0$, we obtain that $\eta \leq \eta (p_0) \leq \sqrt{C_1}$, hence we conclude the inequality $\xi \leq \tilde{\xi} + log \sqrt{C_1}$. We can apply the 
same reasoning to $- \xi$ instead of $\xi$, then we have that, at $w = 0$,
\begin{center}
$|\xi (0)| \leq  \tilde{\xi} (0) + log \sqrt{C_1} \leq \hspace{1pt}  \underset{\mathbb{D}}{\mathrm{sup}} \hspace{1.5pt} log (2 \alpha) + log \sqrt{C_1} =: C_2$,
\end{center}
and this implies that $|\xi (z)| \leq C_2$ if $|z| > max\{ c_1 ^{-1} , R'_1\}$.

Take $z \in \mathbb{C}$ such that $|z| \geq max\{ r/c_1 , R'_1 \}$. Using Euclidean coordinates $x + iy$ in $D_{|\phi|}(z , r)$, define the function 
$\Psi : D_{|\phi|}(z , r) \rightarrow \mathbb{R}$ as 
\begin{center}
$\Psi(x, y) = \frac{\displaystyle C_2}{\displaystyle cosh r} cosh(\sqrt{2}x) cosh(\sqrt{2}y)$,  
\end{center}
we have $\Delta_0 \Psi = 4 \Psi$, and $\Psi \geq C_2 \geq \xi$ in $\partial D_{|\phi|} (z , r)$. Moreover, $\Psi \geq \xi$ in $D_{|\phi|} (z , r)$.
In fact, if $\Psi - \xi$ admits a negative minimum at $p_0$, it would be in the interior of the disc, therefore $\xi(p_0) > \Psi(p_0) \geq 0$ and 
$\Delta_0 (\Psi - \xi) (p_0) \geq 0$. On the other hand, we have at $p_0$ that  
\begin{center}
$\Delta_0 (\Psi - \xi) = 4\Psi + 2K_{\mathbb{M}}sinh (2 \xi) \leq 4(\Psi + K_{\mathbb{M}} \xi) \leq 4 max \{ 1, b^2 \} (\Psi - \xi) < 0$,
\end{center}
a contradiction. Analogously, $\Psi \geq - \xi$, and then $\Psi \geq |\xi|$. Therefore, evaluating at $z$, $|\xi(z)| \leq \frac{C_2}{cosh r}$. Consequently, we conclude that
\begin{align}\label{expdecay}
|\xi(z)| \leq \displaystyle 2 C_2 \displaystyle e^{-c_1 |z|}. 
\end{align}

This estimate implies that $|\xi| \rightarrow 0$ at the punctures. Consequently, the tangent planes become vertical at infinity.

\bfseries Remark. \mdseries It is easy to verify that, for any $\epsilon \in (0 , 1)$, there exists $\delta = \delta(\epsilon)$ and $R = R (\epsilon)$ such that the disc 
$D_{|\phi|}(z , \delta|z|^{m_j + 1})$ is contained in $D(z , \epsilon |z|)$, for all $z \in \mathbb{C}$ satisfying $|z| > R$. 

\item We finally prove the last statement. Recall that we parametrized an end representative of $p_j$ (denoted by $E_j$) by 
$A(R_j)$ such that the Hopf differential of $h$ has the expression $((m_j + 1)z^{m_j} + \frac{ci}{z})^{2}dz^2$, for some $c \in \mathbb{R}$. 
Without loss of generality, we can assume that 
\begin{align}\label{radiusassump}
R_j^{m_j + 1} > 1 + (4 \pi |c|/cos(\pi/10)).  
\end{align}
Then, we can locally define the map $F(z) := \int \sqrt{\phi(z)} dz = \int (m_j + 1)z^{m_j} + \frac{ci}{z}dz$. It is clear that $Im F$ is globally well defined, and if $\theta$ is 
a locally defined argument function, we have
\begin{align*}
 Im F (z) = clog|z| + |z|^{m_j + 1} sin((m_j + 1)\theta)
\end{align*}
and, locally,
\begin{align*}
 Re F (z) = -c\theta + |z|^{m_j + 1} cos((m_j + 1)\theta).
\end{align*}

From now on, we denote by $F_{\Omega}$ a branch of $F$ defined on the domain $\Omega \subset A(R_j)$. 

Consider now the following concept:

\bfseries Definition. \mdseries Given a piecewise smooth curve $\gamma: [0 , l] \rightarrow \mathbb{C}$, a \textit{generalized lift} of $\gamma$ is a piecewise smooth curve 
$\beta: [0 , l] \rightarrow A(R_j)$ such that there exists a partition $0 = t_0 < t_1 < \cdots < t_{n+1} = l$, for some $n \in \mathbb{N}$ and domains 
$D_i \subset A(R_j)$, $i = 0 , \cdots , n$, where we can define a branch of the logarithm, such that 
\begin{itemize}
 \item $\beta([t_i , t_{i+1}]) \subset D_i$, $i = 0, \cdots , n$;
 \item $\gamma$ is the juxtaposition of the paths $F_{D_0}(\beta|_{[t_0 , t_1]}), \cdots , F_{D_n}(\beta|_{[t_n , t_{n+1}]})$, in this order.
\end{itemize}

This result is crucial for the proof:
\begin{lemma}\label{construction}
Fix $C > 0$. Let $\gamma^C_1 : [0 , 8C] \rightarrow \mathbb{C}$ be the curve given by

\[  \gamma^C_1(t) =  \left\{
\begin{array}{ll}
      C + it, &  t \in [0 , C]; \\
      2C - t + iC, & t \in [C , 3C]; \\
      -C + i(4C - t) & t \in [3C , 5C]; \\
      t -6C -iC & t \in [5C , 7C]; \\
      C + i(t - 8C) & t \in [7C , 8C]. \\
\end{array} 
\right. \]
Let also $\gamma^C : [0 , 8(m_j + 1)C] \rightarrow \mathbb{C}$ be the curve $\gamma^C_1$ traversed $m_j + 1$ times. Then, for $C$ sufficiently large, the curve $\gamma^C$ admits a 
generalized lift $\tilde{\gamma}^C$ which starts and finishes at the same connected component of $(Im F)^{-1}(0)$.
\end{lemma}
\begin{proof}
Suppose first that $c=0$. Hence $F: A(R_j) \rightarrow A(R_j^{m_j + 1})$ is a well-defined covering map, and if $C > R_j^{m_j + 1}$, it is enough to take the usual lift
of $\gamma^C$.

Now, suppose $c$ is nonzero. It is known (see [HNST]) that, if $R_j$ is large enough, the set $(Im F)^{-1}(0)$ consists of $2(m_j + 1)$ connected components, denoted by 
  $l_0, \cdots , l_{2m_j +1}$, and each of them is a smooth curve whose boundary is a point in $\{z; |z| = R_j \}$ and, for each $k \in \{ 0 , \cdots , 2m_j + 1 \}$, the curve $l_k$
is contained in the set 
\begin{align*}
\{z \in A(R_j); \frac{k \pi}{(m+1)} - \frac{\pi}{10(m+1)} < arg(z) < \frac{k \pi}{(m+1)} + \frac{k \pi}{10(m+1)}\}.
\end{align*}

Consider $M_0 := R_j^{m_j + 1} + 4\pi |c|$ and $M_1 := max \{ |ImF(z)|; |z| = R_j\}$. Choose $C > max \{ M_0 , M_1\}$. For each
$k = 0 , \cdots , 2m_j + 1$, let $\Delta_k$ be the domain 
\begin{align*}
 \{ |z| > R_j \, and \, \frac{k\pi}{m_j +1} - \frac{\pi}{10(m_j + 1)} < arg(z) < \frac{(k + 1)\pi}{m_j + 1} + \frac{\pi}{10(m_j + 1)}\},
\end{align*}
and let $\Omega_k$ be the (open) subdomain of $\Delta_k$ bounded by $l_{k}$, $l_{k+1}$ and $\{z ; |z| = R_j \}$ (here, $l_{2m_j + 2} = l_0$). We can consider an argument 
function in $\Delta_k$ taking values in the interval $(\frac{k\pi}{m_j +1} - \frac{\pi}{10(m_j + 1)} , \frac{(k + 1)\pi}{m_j + 1} + \frac{\pi}{10(m_j + 1)})$, then we can define 
$F_k$ as $F_{\Delta_k}$. The assumption (\ref{radiusassump}) implies that $Re F_k(z) > 0$ if $z \in l_{2k}$. 
In fact, when $z \in l_{2k}$, we have 
\begin{center}
$Re F_k(z) = |z|^{m_j + 1}cos[(m_j + 1)arg z] - arg z c \geq R_j^{m_j+1}cos(\pi/10) - 4\pi |c| > 0$. 
\end{center}
The same argument proves that $Re F_k(z) < 0$ if $z \in l_{2k+1}$. Since $\phi$ never vanishes in $A(R_j)$ (we can choose $R_j$ to be large enough), 
the derivative of $Re F_k$ is never zero along $l_k$. Hence, in particular, since $C > M_0$, there is a unique point $p \in l_0$ such that $F_0(p) = C$.

In order to construct $\tilde{\gamma}^C$, the first step is to obtain a (usual) lift of $\gamma^C|_{[0, 4C]}$ with respect to $F_0 : \Delta_0 \rightarrow \mathbb{C}$. Consider 
the number
\begin{center}
$C_0 := sup \{ t \in [0 , 4C] ; \exists \beta_t : [0 , t] \rightarrow \overline{\Omega}_0 , \beta_t (0) = p\, \text{and}\, F_0 \circ \beta_t = \gamma^C|_{[0 , t]}  \}$.
\end{center}
Since $\phi$ does not have zeroes in $A(R_j)$, by the Inverse Function Theorem and the fact that $F_0$ preserves orientation, there exists a path 
$\beta_{\delta} : [0 , \delta] \rightarrow \overline{\Omega}_0$ satisfying $\beta_{\delta} (0) = p$ and $F_0 \circ \beta_{\delta} = \gamma^C|_{[0 , \delta]}$, for some $\delta > 0$. 
Hence $C_0 > 0$. Moreover, we can define a lift $\hat{\beta} : [0 , C_0) \rightarrow \overline{\Omega}_0$ of $\gamma^C|_{[0 , C_0)}$. 

Now, we prove that we can extend $\hat{\beta}$ to $[0 , C_0]$, taking values in $\overline{\Omega}_0$. In order to do this, take a sequence $(t_n)_{n \in \mathbb{N}}$ in $[0 , C_0)$ converging to $C_0$. We know that 
either $|Re F_0(\hat{\beta}(t_n))| = C$, for all $n$, or $Im F_0(\hat{\beta}(t_n)) = C$, for all $n$, up to taking a subsequence. 

Suppose that $(|\hat{\beta}(t_n)|)_{n}$ escapes to infinity. Since the sequence 
\begin{center}
$(F_0(\hat{\beta}(t_n)))_n$  
\end{center}
is bounded, by the expression of $F_0$, we conclude that 
$(\hat{\beta}(t_n) / |\hat{\beta}(t_n)|)_{n}$ converges to $(0,0)$, a contradiction. It proves that, for any sequence $(t_n)_{n \in \mathbb{N}}$ in $[0 , C_0)$ converging to 
$C_0$, the sequence $(\hat{\beta}(t_n))_{n}$ has an accumulation point in $\overline{\Omega}_0$ (up to taking a subsequence, we can suppose that $(\hat{\beta}(t_n))_{n}$ converges). Suppose 
$(\hat{\beta}(t_n))_n$ converges to a point in $\{z ; |z| = R_j \}$. If $|Re F_0(\hat{\beta}(t_n))| = C$, for all $n$, we have that 
\begin{center}
$C \leq |arg(\hat{\beta}(t_n)) c| + |\hat{\beta}(t_n)|^{m_j + 1} \leq 2 \pi |c| + |\hat{\beta}(t_n)|^{m_j + 1}$, 
\end{center}
and taking limits, we conclude that $C \leq 2 \pi |c| + R_j^{m_j + 1} < M_0$, a contradiction. Since $C > M_1$, it is not possible that $Im F_0(\hat{\beta}(t_n)) = C$, for all 
$n$, therefore $(\hat{\beta}(t_n))_n$ does not converge to a point in $\{z ; |z| = R_j \}$.

If the sequence $(\hat{\beta}(t_n))_{n}$ converges to a point $q \in \Omega_0$, by continuity, we have that 
$F_0(q) \in \{z \in \mathbb{C} ; max\{ |Re F_1(z)| , |Im F_1 (z)| \} = C \}$ and that $\gamma^C(C_0) = F_0 (q)$. Taking a neighborhood $U \subset \Omega_0$ of $q$ such that 
${F_0}|_U$ is a diffeomorphism onto its image, there exists $\delta > 0$ such that $\gamma([C_0 - \delta , C_0 + \delta]) \subset U$. Therefore, we can define 
$\beta_{C_0 + \delta} : [0 , C_0 + \delta] \rightarrow \Omega_0$ as

\[  \beta_{C_0 + \delta}(t) =  \left\{
\begin{array}{ll}
      \hat{\beta}(t), &  t \in [0 , C_0); \\
      F_0^{-1}(\gamma^C(t)), & t \in (C_0 - \delta , C_0 + \delta], \\
\end{array} 
\right. \]

and we deduce that $C_0 + \delta \leq C_0$, a contradiction. 

It remains to analyze the case when $(\hat{\beta}(t_n))_{n}$ converges to a point $q$ in $l_0 \cup l_1$. Clearly, $F_0$ is defined at $q$ and it is a real number. If $q \in l_0$,
since $|Re F_0(\hat{\beta}(t_n))| = C$ and $Re F_0 > 0$ along $l_0$, we conclude that $Re F_0(q) = C$, then $p=q$. Proceeding as before, we can smoothly extend $\hat{\beta}$
to $\beta_{C_0} : [0 , C_0] \rightarrow \overline{\Omega}_0$. Since $\beta_{C_0}$ is a lift of $\gamma^C|_{[0 , C_0]}$ and $C_0 \leq 4C$, we obtain that $0 < C_0 \leq C$. 
Furthermore, $Im F_0(\beta_{C_0})$ must be strictly increasing along $[0 , C_0]$, but $Im F_0(\beta_{C_0}) (C_0) = Im F_0(\beta_{C_0}) (0)$, a contradiction. Therefore, 
$q \in l_1$ and $C_0 = 4C$. Therefore, $F_0 (4C) = -C$, and we obviously are able to smoothly extend $\hat{\beta}$ to $\beta_{4C} : [0 , 4C] \rightarrow \overline{\Omega}_0$.

Inductively, for $k = 1, \cdots , 2m_j + 1$, we have a curve $\beta_{4kC} : [4kC , 4(k+1)C] \rightarrow \overline{\Omega}_k$ starting at $\beta_{4(k-1)C}(4kC)$, lifting 
$\gamma^C|_{[4kC , 4(k+1)C]}$ with respect to $F_k : \Delta_k \rightarrow \mathbb{C}$, and $\beta_{4kC} (4(k+1)C) \in l_{k+1}$. Finally, we define $\tilde{\gamma}^C$ as the 
juxtaposition of $\beta_{4C} , \cdots , \beta_{8(m_j + 1)C}$, in this order. Evidently, the point $\tilde{\gamma}^C(8(m_j + 1)C)$ is in $l_{2m_j + 2} = l_0$, as well as 
$\tilde{\gamma}^C(0)$.
\end{proof}

A consequence of the arguments of the preceding proof is that we can cover $A(R_j)$ by domains $\Delta_k$, $k = 0 , \cdots , 2m_j + 1$, where we can 
define an integral of $\sqrt{\phi}$, denoted by $F_k : \Delta_k \rightarrow \mathbb{C}$ (the domains $\Delta_k$ from Lemma \ref{construction} can also be considered in
$A(R_j)$ when $c = 0$). Since the argument functions used to define the maps $F_k$ are bounded from above by $4\pi$, in absolute value, we have that there exist $R^*, C_* > 0$ 
independent on $k$ such that, when $|z| > R^*$, the following inequality holds:
\begin{align*}
C_* |z|^{m_j + 1} > |F_k(z)| > C_*^{-1} |z|^{m_j + 1}.
\end{align*}

Let $P(C , p_j )$ be the curve obtained from $\tilde{\gamma}^C$ when we connect $\tilde{\gamma}^C(0)$ and $\tilde{\gamma}^C(8(m_j + 1)C)$ by the shortest curve 
segment in $l_0$. 

We list some properties of $P(C , p_j )$, whose proofs can be deduced by the arguments in the demonstration of Lemma \ref{construction}.
\begin{enumerate}
 \item $P(C , p_j )$ is a simple, piecewise smooth closed curve. If $c = 0$, it has $4(m_j + 1)$ vertices, all of them having internal angle $\frac{\pi}{2}$; if 
 $c \neq 0$, it has $4(m_j + 1) + 2$ vertices, one of them having internal angle $\frac{3\pi}{2}$, and the other ones having internal angle $\frac{\pi}{2}$.
 \item The bounded region determined by $P(C , p_j )$ contains $D(0 , R_j)$.
 \item Given any compact set $K$ in $\mathbb{C}$, there exists $\tilde{C}$ such that $P(C , p_j )$ does not intersect $K$, $\forall C > \tilde{C}$. 
 \end{enumerate}

For $k = 0 , \cdots , m_j$ and $l = 0, 1$, let $A^l_k (C)$ be the arc 
\begin{center}
$\tilde{\gamma}^C([(8k + 4l + 1)C ,(8k + 4l + 3)C])$.  
\end{center}
By construction, $A^l_k (C)$ is bijectively mapped onto a subset of $\{ w \in \mathbb{C} ; |Im w| = C \}$ by the map $F_{2k + l}$. Let also 
$B^l_k (C)$ be the arc 
\begin{center}
$\tilde{\gamma}^C([(8k + 4l)C , (8k + 4l + 1)C]) \cup \tilde{\gamma}^C([(8k + 4l + 3)C , (8k + 4l + 4)C])$, 
\end{center}
for $k = 0 , \cdots , m_j$ and $l = 0, 1$. 
Each of these curves are one-to-one mapped onto a subset of $|\{ w \in \mathbb{C} ; |Re w| = C \}$ by the map $F_{2k +l}$. Denote by $B^* (C)$ the (possibly degenerate) 
compact arc of $P(C , p_j )$ which connects $\tilde{\gamma}^C(0)$ and $\tilde{\gamma}^C(8(m_j + 1)C)$ 
lying in $l_0$. We are going to denote by $\mathcal{I}(C)$ and $\mathcal{R}(C)$ the union of the curves $A^l_k (C)$ and $B^l_k (C)$, respectively. It is true that there is 
a small neighborhood $V$ of $A^l_k (C)$ contained in $\Omega_{2k + l}$ such that $F_{2k+l} : V \rightarrow F_{2k+l}(V)$ is a conformal diffeomorphism. A similar property holds 
for the curves $B^l_k (C)$ and for the curve $B^* (C)$.

We now proceed to the proof. Applying Gauss-Bonnet on $S(r)$, we obtain that
\begin{align}\label{gbcompact}
 \displaystyle \int_{S(r)} K_{\Sigma} dA + \int_{\partial S(r)} \kappa_g = 2\pi (2 - 2g - n).
\end{align}

Consider, in the $z$-plane, the annulus $\Omega(C , R_j, p_j )$ in $\mathbb{C}$ bounded by the union of two curves: the circle $\{ |z| = R_j \}$ and the curve 
$P(C , p_j )$. Again, by Gauss-Bonnet, we have
\begin{align}\label{gbpart}
\int_{\Omega(C , R_0, p_j )} K_{\Sigma} dA + \int_{P(C , p_j )} \kappa_g - \int_{\{ |z| = R_j \}} \kappa_g = -2\pi (m_j + 1). 
\end{align}

Summing Equation (\ref{gbcompact}) with the equations in (\ref{gbpart}) for all $j$, we obtain
\begin{align*}
\int_{\tilde{S}(C)} K_{\Sigma} dA + \sum_{j=1}^{n} \int_{P(C , p_j )} \kappa_g = 2\pi (2 - 2g - 2n - \sum_{j=1}^{n} m_j), 
\end{align*}
where $\tilde{S}(C) = S(r) \cup [\bigcup^{n}_{j=1} \Omega(C , R_j, p_j )]$. As $C$ goes to infinity, $\tilde{S}(C)$ goes to $S^{*} \cong \Sigma$. It is enough to
prove that $\int_{P(C , p_j )} \kappa_g$ goes to zero as $C$ goes to $+ \infty$.

For each $k \in \{ 0 , \cdots , 2m_j + 1 \}$, we know that $ImF^{-1}(0) \cap \Omega_k$ is at a positive distance from the lines that bound $\Delta_k$. Then, there
exist positive numbers $\delta_k$, $\epsilon_k$ and $R^{*}_k$ such that $D_{|\phi|}(z , \delta_k |z|^{m_j + 1}) \subset D(z , \epsilon_k |z|)$, for all $z \in A(R_k)$ satisfying 
$|z| > R^{*}_k$. Moreover, choosing $\epsilon_k$ to be small enough, we can assure that $D(z , \epsilon_k |z|) \subset \Delta_k$ when $z \in \Omega_k$ and $|z| > R^{*}_k$. 

If $\epsilon^{(0)} := min \{ \epsilon_0 , \cdots , \epsilon_{2m_j +1} \}$, we take $R^{*} > 0$ such that, if $|z| > R^{*}$, the following properties hold:
\begin{itemize}
 \item $D_{|\phi|}(z , 1) \subset D(z , \epsilon^{(0)} |z|) \subset \Delta_k$, for some $k$ depending on $z$;
 \item $F_k: D_{|\phi|}(z , 1) \rightarrow F(D_{|\phi|}(z , 1))$ is a conformal diffeomorphism;
 \item $C_* |z|^{m_j + 1} > |F_k(z)| > C_*^{-1} |z|^{m_j + 1}$;
 \item There exist positive constants $\widehat{C}$ and $\hat{c}$, not depending on $k$, such that 
 \begin{align*}
 \underset{D_{|\phi|}(z , 1)}{\mathrm{sup}} \hspace{1.5pt} |\xi| \leq \displaystyle \widehat{C} \displaystyle e^{-\hat{c} |z|};
 \end{align*}
 \item $\underset{D_{|\phi|}(z , 1)}{\mathrm{sup}} \hspace{1.5pt} cosh (2\xi) \leq 2$.
\end{itemize}

We can consider $w$-coordinates in $D_{|\phi|}(z , 1)$ induced by $F_k$, $k$ depending on $z$ (notation: $w := F_k(z)$); in these parameters, the function $\xi$ satisfies the 
equation
\begin{align}\label{normgordon}
\Delta_{|\phi|} \xi = -2K_{\mathbb{M}}sinh (2 \xi).
\end{align}

If $z$ satisfies $|z| > R^{*}$, define $B_1(z)$ as $D_{|\phi|}(z , 1)$. By Theorem 3.9 of [GT], we can conclude the following 
interior gradient estimate for the Poisson equation:
\begin{align*}
\underset{B_1(z)}{\mathrm{sup}} \hspace{1.5pt} ||\nabla \xi|| \leq \widetilde{C}(\underset{B_1(z)}{\mathrm{sup}} \hspace{1.5pt} |\xi| +  \underset{B_1(z)}{\mathrm{sup}} \hspace{1.5pt}|2K_{\mathbb{M}}sinh (2 \xi)|),
\end{align*}
for a universal constant $\widetilde{C}$. Since $\underset{B_1(z)}{\mathrm{sup}} \hspace{1.5pt} cosh (2\xi) < 2$, we obtain that
\begin{center}
$\underset{B_1(z)}{\mathrm{sup}} \hspace{1.5pt}|sinh (2 \xi)| < 4\underset{B_1(z)}{\mathrm{sup}} \hspace{1.5pt} |\xi|$. 
\end{center}
Therefore, we have the estimate
\begin{align*}
\underset{B_1(z)}{\mathrm{sup}} \hspace{1.5pt} ||\nabla \xi|| \leq 9 \widetilde{C} max \{ 1 , a^2\} \underset{B_1(z)}{\mathrm{sup}} \hspace{1.5pt} |\xi|.
\end{align*}
By the properties stated above, we rewrite the estimate as
\begin{align*}
\underset{B_1(z)}{\mathrm{sup}} \hspace{1.5pt} ||\nabla \xi|| \leq \widetilde{C}  e^{-\tilde{c} |w|^{(m_j +1)^{-1}}},
\end{align*}
renaming $9 \widetilde{C} max \{ 1 , a^2\} \widehat{C}$ by $\widetilde{C}$, for simplicity. Clearly, $9 \widetilde{C} max \{ 1 , a^2\} \widehat{C}$ does not depend on $k$. In 
particular, we conclude that
\begin{align}\label{expdecaygrad}
||\nabla \xi (w)|| \leq \widetilde{C}  e^{-\tilde{c} |w|^{m'_j}},
\end{align}
for $m'_j := (m_j +1)^{-1}$.

First, let us prove that $\int_{\mathcal{I}(C)} \kappa_g ds$ goes to $0$ as $C$ goes to $+\infty$. Fixing a curve $A^0_k (C)$ in $\mathcal{I}(C)$, 
we know that this curve can be parametrized as $\tau_{C}(x) = x + iC$, for $x \in [-C , C]$. Based on [H], we have that 
\begin{center}
$ k_{\tau_{C}} =  - \frac{ \displaystyle  \xi_y }{\displaystyle 2 cosh (\xi)}$,  
\end{center} 
where $k_{\tau_{C}}$ is the geodesic curvature of $\tau_{C}$ as a curve in $\mathbb{M} \times \mathbb{R}$.

Along the curve $\tau_C$, we have that, when $|w|$ is sufficiently large, by the estimate in (\ref{expdecaygrad}), 
\begin{align*}
 |\xi_y (w)| & \leq ||\nabla \xi(w)|| \leq \displaystyle \widetilde{C} \displaystyle e^{-\tilde{c}|w|^{m'_j}} \leq \widetilde{C} e^{-\tilde{c}_1 (|x|^{m'_j} + |C|^{m'_j})},  
\end{align*}
for positive constants $\widetilde{C}, \tilde{c}$ and $\tilde{c}_1$. Therefore, we have
\begin{align*}
 \int_{\tau_{C}} |\kappa_g| ds & \leq \int_{\tau_{C}} |\kappa_{\tau_{C}}| ds = \int_{0}^{C} |\xi_y| dx \\
 & \leq \widetilde{C} \int_{0}^{+ \infty} e^{-\tilde{c}_1 (|x|^{m'_j} + |C|^{m'_j})} dx \\
 & \leq \widetilde{C} e^{-\tilde{c}_1 |C|^{m'_j}} \int_{- \infty}^{+ \infty} e^{-\tilde{c}_1 |x|^{m'_j}} dx,
\end{align*}
and the last term certainly goes to zero as $C$ goes to $+ \infty$. The same argument can be applied to $A^1_k (C)$, and then we conclude that $\int_{\mathcal{I}(C)} \kappa_g ds$
converges to zero as $C$ goes to $+ \infty$.

Now, we are going to prove that 
\begin{center}
$\int_{\mathcal{R}(C)} \kappa_g \rightarrow 0$ as $C \rightarrow +\infty$. 
\end{center}
As in the previous case, we are going to compute the curvature of $\chi_C(y) = (h(C , y) , 2 y)$ as a curve in $\mathbb{M} \times \mathbb{R}$. We have 
$|h_y(C , y)|^2 = 4 sinh^2\xi(C , y)$ and $\chi'_C = X_y$. It is clear that $\overrightarrow{K}_{\chi_C}$, a curvature vector of $\chi_C$, has the expression
\begin{align*}
\overrightarrow{K}_{\chi_C} = \frac{1}{4 cosh^2 (\xi)}\nabla_{X_y} X_y - \frac{\xi_y sinh (\xi)}{4 cosh^3 (\xi)} X_y,
\end{align*}
$\nabla$ the Levi-Civita connection of $\mathbb{M} \times \mathbb{R}$.

We will decompose $\overrightarrow{K}_{\chi_C}$ in terms of the frame $\{ X_x , X_y , N \}$. It is easy to verify that 
\begin{center}
$N = \left( \frac{\displaystyle h_y}{\displaystyle 2 cosh (\xi) sinh (\xi)} , tanh (\xi) \right)$.
\end{center}
By simple computations, we have the identity
\begin{align*}
 \nabla_{X_y} X_y = -tanh (\xi) \xi_x X_x +  tanh (\xi) \xi_y X_y - 2 \xi_y N, 
\end{align*}
thus we have
\begin{align*}
\overrightarrow{K}_{\chi_C} = - \frac{\xi_x senh (\xi)}{4 cosh^3 (\xi)} X_x - \frac{\xi_y }{2 cosh^2 (\xi)} N.
\end{align*}
Finally, we conclude that, along $\chi_C$,
\begin{align*}
 |\kappa_g|^2 \leq |\overrightarrow{K}_{\chi_C}|^2 \leq \frac{\xi^2_x senh^2 (\xi) + \xi^2_y}{4 cosh^4 (\xi)} \leq \frac{||\nabla \xi||^2}{ 4 cosh^2 (\xi)}, 
\end{align*}
and the conclusion follows as in the first case.

We finally prove that $\int_{B^* (C)} \kappa_g \rightarrow 0$ as $C \rightarrow +\infty$. Using the $w$-coordinates induced by $F_0$, we have that $B^* (C)$ is contained in the 
real interval $[C - + 2\pi |c|, C + 2\pi |c|]$ of the $w$-plane. Proceeding exactly as in the first case, we obtain the estimate 

\begin{align*}
 \int_{B^* (C)} |\kappa_g| ds  \leq \int_{C - 2\pi |c| }^{C + 2\pi |c|} |\xi_y| dx \leq \widetilde{C} \int_{C - 2\pi |c|}^{+ \infty} e^{-\tilde{c}_1 |x|^{m'_j} } dx, 
\end{align*}
 
 which goes to $0$ as $C$ goes to $+ \infty$.
\end{enumerate}
\end{proof}
\begin{comment}By the Schauder's estimates, we have that
\begin{align*}
 |\xi|_{2 , \alpha} & \leq C(|\Delta \xi|_{0, \alpha} + |\xi|_0) \\
 & \leq C(|K_{\mathbb{M}} sinh \xi|_{0, \alpha} + |\xi|_0) \\ 
 & \leq C(|sinh \xi|_{0, \alpha} + |\xi|_0)) \\
 & \leq C e^{-R}.
\end{align*}
The rest of this case follows as in [HR]. Justifying the next to last inequality, notice that $|K_{\mathbb{M}} sinh \xi|_{0} \leq |K_{\mathbb{M}}|_{0} |sinh \xi|_{0}$, and
that $|K_{\mathbb{M}}(z) sinh \xi (z) - K_{\mathbb{M}}(z_1) sinh \xi (z_1)$ (crucial mistake here)

We then prove that $\int_{|Im(z^{m_i + 1})| = C_1} \kappa_g$ goes to $0$ as $C_1$ goes to $+\infty$.
\end{comment}
We state below an easy consequence of Theorem \ref{main} (see [PR]):
\begin{corollary}\label{profile}
Let $S$ be a complete, orientable minimal surface with finite total curvature, and $p$ an end of $S$. If $m_p \geq 0$ is the integer associated to $p$, then $p$ corresponds
to $m_p + 1$ geodesics $\gamma_1 ,..., \gamma_{m_p+1} \subset \mathbb{M}^2 \times \{ + \infty \}$, $m_p + 1$ geodesics 
$\Gamma_1 ,..., \Gamma_{m_p+1} \subset \mathbb{M}^2 \times \{ - \infty \}$, and $2(m_p + 1)$ vertical straight lines (possibly some of them coincide) in 
$\partial_{\infty} \mathbb{M}^2 \times \mathbb{R}$, each one joining an endpoint of some $\gamma_j$ to an endpoint of some $\Gamma_j$.
\end{corollary}
The \textit{asymptotic profile} of the end $p$ is the asymptotic polygon formed by the curves $\gamma_i$, $\Gamma_j$ and the vertical straight lines described in Corollary 
\ref{profile}. 
% When the end is embedded, the coincidences mentioned in the corollary do not happen. 
\section{Examples}\label{example}
In this section, we give some examples of minimal surfaces with finite total curvature in $\mathbb{M} \times \mathbb{R}$.
\begin{enumerate}
 \item \textit{Vertical planes}. The simplest examples are the vertical totally geodesic planes $\alpha \times \mathbb{R}$, where $\alpha$ is a horizontal geodesic. Their 
 total curvature is zero, and these are the only surfaces satisfying this condition. In fact, let $\Sigma$ be a minimal surface with vanishing total curvature. The Gauss 
 equation states that
\begin{center}
$K_{\Sigma} = K_{\mathbb{M} \times \mathbb{R}}|_{G(\Sigma)}+ K_{ext}$, 
\end{center} 
where $K_{\Sigma}$ and $K_{ext}$ are the intrinsic and extrinsic curvatures of $\Sigma$, respectively, and $K_{\mathbb{M} \times \mathbb{R}}|_{G(\Sigma)}$ is the sectional 
curvature of the ambient restricted to the Grassmanian of tangent planes of $\Sigma$. The curvature $K_{\Sigma}$ is nonpositive, by the minimality of $\Sigma$, and thus 
$K_{\Sigma}$ is identically zero, since the total curvature vanishes. It implies that $K_{\mathbb{M} \times \mathbb{R}}|_{G(\Sigma)} \equiv K_{ext} \equiv 0$, therefore 
$\Sigma$ is a totally geodesic surface whose tangent planes are always vertical. Finally, given a vertical plane $P \in T_{(p , r)} (\mathbb{M} \times \mathbb{R})$, there 
is exactly one totally geodesic surface in $\mathbb{M} \times \mathbb{R}$ that is tangent to $P$, which is $\gamma_v \times \mathbb{R}$, where $\gamma_v$ is the geodesic of
$\mathbb{M}$ satisfying $\gamma'_v (0) = v \in (T_p \mathbb{M} \times \{ 0 \}) \cap P$, $v \neq 0$, and the assertion is proved. 

\item \textit{Scherk graphs}. Let $P$ an ideal geodesic polygon in $\mathbb{M}$ whose vertices are the points of infinity $p_1, \cdots , p_{2n} \in \partial_{\infty} \mathbb{M}$.
 Denote by $A_i$ the complete geodesic connecting $p_{2i-1}$ to $p_{2i}$, $i = 1, \cdots , n$, and by $B_i$ the complete geodesic connecting $p_{2i}$ to $p_{2i + 1}$, 
 $i = 1, \cdots , n$, where $p_{2n+1} := p_1$. 

 Consider the family $\mathcal{H} = \{ H_i \}_{i=1}^{2n}$, where for each $i = 1, \cdots , 2n$, $H_i$ is a horocycle at $p_i$ bounding an open horodisc $F_i$ such that 
 $H_i \cap H_j = \emptyset$ if $i \neq j$. Denote by $\tilde{A}_i$ the geodesic segment given by $A_i \backslash (\cup_{j=1}^{2n} F_j)$, and define $\tilde{B}_i$
 in a similar way.
%  $|A_i|$ (resp. $|B_i|$) the distance between the horocycles $H_{2i-1}$ and $H_{2i}$ (resp. $H_{2i}$ to $H_{2i + 1}$). Notice 
%  that $|A_i|$ is precisely the length of the geodesic segment contained in $A_i$ which is outside the horocycles $H_{2i-1}$ and $H_{2i}$, and a similar property is satisfied by 
%  $|B_i|$. 
 Let $\gamma(i)$ be the geodesic segment connecting the two interior points of $H_i \cap P$ and denote by $P(\mathcal{H})$ the polygon 
 \begin{align*}
  \bigcup\limits_{i=1}^{n} (\tilde{A}_i \cup \tilde{B}_i) \cup \bigcup\limits_{j=1}^{2n} \gamma(j)
 \end{align*} 
 and $D(\mathcal{H})$ the domain bounded by $P(\mathcal{H})$.
 
 For a positive $r>0$, let $G(r, \mathcal{H})$ be the graph of the minimal surface equation over $D(\mathcal{H})$ whose boundary data are given by $r$ on 
 $\cup_{i=1}^{n} \tilde{A}_i$ and zero elsewhere on $P(\mathcal{H})$.
 
 Define the following quantities:
 \begin{align*}
  a(P) & = \sum_{i=1}^n |\tilde{A}_i|;\\
  b(P) & = \sum_{i=1}^n |\tilde{B}_i|.
 \end{align*}
Let $D$ the domain bounded by $P$. We say that a geodesic convex polygon $Q$ is \textit{inscribed} in $D$ if the set of vertices of $Q$ is contained in the set of vertices of $P$. 
Using the horocycles $H_i$, define
\begin{align*}
  a(Q) & = \sum_{\tilde{A}_i \subset Q} |\tilde{A}_i|;\\
  b(Q) & = \sum_{\tilde{B}_i \subset Q} |\tilde{B}_i|.
 \end{align*}
 We also define $|Q|$ as the sum of the lengths of the geodesic segments contained in the sides of $Q$ and determined by the horocycles $H_i$.
 In [GR], the authors proved the following theorem:
 \begin{theorem} \label{galvez} There is a solution to the Dirichlet problem for the minimal surface equation in the domain $D$ bounded by $P$ with prescribed data $+\infty$ at 
 $A_i$ and $-\infty$ at $B_i$ if and only if the following two conditions are satisfied:
\begin{enumerate}
\item $a(P) - b(P) = 0$,
\item For all inscribed polygons $Q$ in $D$ different from $P$ there exist horocycles at the vertices such that
\begin{center}
$2a(Q)<|Q|$ and $2b(Q)<|Q|$.
\end{center}
\end{enumerate}
Moreover, the solution is unique up to additive constants.
 \end{theorem}
 The graph of the function described in the theorem are called the \textit{Scherk graph} over $D$. 
 
 By the proof of Theorem \ref{galvez} (see [CR] and [GR]), the Scherk graph $\Sigma_n$ over $D$ is a limit of the sequence of surfaces $(G_k := G(r_k, \mathcal{H}^k))_k$, where 
 $(r_k)_k$ is a sequence going to $+\infty$ as $k$ goes to $+\infty$ and, for each $k$, $\mathcal{H}^k = \{ H^k_i \}_{i=1}^{2n}$ a family of horocycles of $\mathbb{M}$ such that 
 $(H^k_i)_{i=1}^{\infty}$ is a sequence of nested horocycles at $p_i$ converging to this point. Using Gauss-Bonnet on each $G_k$, we conclude that the total curvature of those
 surfaces is uniformly bounded from below by $2\pi(1 - n)$, therefore $\Sigma_n$ has finite total curvature.

 In order to compute explicitly the total curvature of $\Sigma_n$, we notice that, since this surface is a graph, the coincidences mentioned in Corollary \ref{profile} do not 
 happen. Therefore, we have that $m_p = n-1$, following the notation of the same corollary. Consequently, applying the formula of Theorem \ref{main}, we conclude that the total 
 curvature of $\Sigma_n$ is precisely $2\pi(1-n)$.

\item \textit{Horizontal catenoids}. In [P], the author constructs a class of minimal annuli with horizontal slices of symmetry. These catenoids $C$ are similar to the ones 
constructed in [MR] and [Py]. They are limits of compact minimal annuli $(C_n)_{n \in \mathbb{N}}$ whose boundary components $S_n^1$ and $S_n^2$ are contained in the vertical 
planes $P_n^{1}$ and $P_n^{2}$, respectively. Denote by  $\kappa_n^i$, $\hat{\kappa}_n^i$ and $\tilde{\kappa}_n^i$ the geodesic curvatures of $S_n^i$ as a curve of $C_n$, 
$P_n^{i}$ and $\mathbb{M} \times \mathbb{R}$, respectively. Clearly, we have that $\kappa_n^i \leq \tilde{\kappa}_n^i$, and since $P_n^{i}$ is a totally geodesic submanifold of 
$\mathbb{M} \times \mathbb{R}$, the curvatures $\hat{\kappa}_n^i$ and $\tilde{\kappa}_n^i$ are equal, up to a sign. Moreover, for each $i$, the induced metric on $P_n^{i}$ is 
Euclidean, thus the total curvature of $S_n^i$ is $2\pi$. Consequently, using Gauss-Bonnet,
\begin{align*}
 & \int_{C_n} K_{C_n} + \int_{\partial C_n} \kappa_{\partial C_n} = 0 \leftrightarrow \\
 & |\int_{C_n} K_{C_n}| \leq \int_{\partial C_n} |\kappa_{\partial C_n}| \leftrightarrow \\
 & |\int_{C_n} K_{C_n}| \leq \int_{S_1} |\kappa_n^1| + \int_{S_2} |\kappa_n^2| \leftrightarrow \\
 & |\int_{C_n} K_{C_n}| \leq \int_{S_1} |\hat{\kappa}_n^1| + \int_{S_2} |\hat{\kappa}_n^2| = 4\pi.
 \end{align*}
Therefore, $C$ has finite total curvature and its absolute value is at most $4\pi$. On the other hand, by the formula of Theorem \ref{main}, 
\begin{align*}
|\int_{C_n} K_{C_n}| \geq 4\pi, 
\end{align*}
 thus $\int_{C} K_{C} = -4 \pi$.
\end{enumerate}
% \section{Classification results}

\section{Appendix}\label{appendix}
Here, we provide a detailed discussion about some basic results which are useful along the text.
\subsection{Vertical annuli in $\mathbb{M} \times \mathbb{R}$}
In this subsection, we study complete vertical rotational minimal catenoids in $\mathbb{H}^2 \times \mathbb{R}$. We prove that, when suitably placed in 
$\mathbb{M} \times \mathbb{R}$, their mean curvature vector fields do not vanish at any point. We also prove that, for a fixed point $p \in \mathbb{M}$ and \ positive number $R>0$,
there exists a positive number $h = h(p,R)$ such that there is no minimal annulus whose boundary is contained in the set $B_R(p) \times \{ -h' , h' \}$ for $h' > h$, where 
$B_R(p)$ is the open ball of radius $R$ centered in $p$.
\subsubsection{Comparing geometries}
Around a point of $\mathbb{M}$, we consider polar coordinates $(s, \theta)$ on the surface, and the metric is given by $ds^2 + Gd\theta^2$, for some positive smooth 
function $G$ of $s$ and $\theta$. In particular, when $\mathbb{M}$ is the hyperbolic space of curvature $-k^2$, $k>0$ (notation: $\mathbb{H}^2(-k^2)$), we have that the 
function $G$ is precisely $G^{(k)} (s, \theta) := sinh^2 (ks)$. 

Let us consider a rotational surface $\Sigma$ in $\mathbb{M} \times \mathbb{R}$. We can parametrize it by $(s, \theta) \mapsto (s, \theta , h(s))$, and the associated 
coordinate frame is ${\bar{\partial}}_s = \partial_s + h'(s)\partial_z$ and ${\bar{\partial}}_{\theta} = {\partial}_{\theta}$ (here, we consider in 
$\mathbb{M} \times \mathbb{R}$ the coordinates $(s, \theta , z)$).  So, the vector field $N = (1 + h'(s)^2)^{-\frac{1}{2}}(-h'(s)\partial_s + \partial_z)$ along $\Sigma$
is normal and unitary, and the mean curvature with respect to it is given by
\begin{align*}
2H = \frac{1}{2G(1 + h'(s)^2)^{\frac{3}{2}}} \left( 2G h''(s) + (1 + h'(s)^2) h'(s) G_s \right).
\end{align*}
Then the surface $\Sigma$ is minimal if and only if 
\begin{center}
$2Gh''(s) + (1 + h'(s)^2)h'(s) G_s = 0$. 
\end{center}
In particular, when $\mathbb{M} = \mathbb{H}^2 (-k^2)$, the equation becomes
\begin{align}\label{eqhip}
sinh(ks)h''(s) + k cosh(ks)(1 + h'(s)^2)h'(s) = 0.
\end{align}
Fix two constants $A, k > 0$ and let $R_{A,k}:= \frac{arcsinh(A)}{k}$. Consider the function $h_{A,k} : [R_{A,k} , + \infty) \rightarrow \mathbb{R}$ defined by
\begin{align*}
 h_{A,k}(s) = \int_{R_{A,k}}^{s} \frac{A}{\sqrt{sinh^2 (kr) - A^2}}dr.
\end{align*}
The following facts about $h_{A,k}$ are easy to verify:
\begin{itemize}
 \item $h_{A,k} \in C^{\infty}((R_{A,k} , + \infty)) \cap C^{0}([R_{A,k} , + \infty))$;
 \item $h_{A,k}$ solves Equation \ref{eqhip} on the domain $(R_{A,k} , + \infty)$; 
 \item $h'_{A,k} > 0$ and $lim_{s \rightarrow R_{A,k}} h'(s) = + \infty$.
\end{itemize}
In $\mathbb{H}^2(-k^2) \times \mathbb{R}$, define the subset 
\begin{align*}
C^{A,k} := \{ (s , \theta , (-1)^{j} h_{A,k}(s)), s \geq R_{A,k} , j \in \{ 0 , 1\} \}. 
\end{align*}
Obviously, $C^{A,k}$ is a complete vertical rotational minimal catenoid in the space $\mathbb{H}^2(-k^2) \times \mathbb{R}$.

We now define, in $\mathbb{M} \times \mathbb{R}$, the surface
\begin{align*}
C_{\mathbb{M}}^{A,k} := \{ (s , \theta , (-1)^{j} h_{A,k}(s)), s \geq R_{A,k} , j \in \{ 0 , 1\} \}, 
\end{align*}
for some fixed polar coordinate system in $\mathbb{M}$. This surface is a complete vertical rotational annulus in $\mathbb{M} \times \mathbb{R}$. If the sectional curvature 
of $\mathbb{M}$ satisfies $-k_1^2 < K_{\mathbb{M}} < -k_2^2$, then, by a slight variation of Proposition 2 of [GL], we have that
\begin{align}\label{lozano}
\frac{\displaystyle G^{(k_1)}_s}{\displaystyle G^{(k_1)}} %\frac{\displaystyle 2k_1 cosh(k_1 s)}{\displaystyle sinh(k_1 s)} 
 > \frac{\displaystyle G_s}{\displaystyle G} > 
 \frac{\displaystyle G^{(k_2)}_s}{\displaystyle G^{(k_2)}}. %\frac{\displaystyle 2k_2 cosh(k_2 s)}{\displaystyle sinh(k_2 s)}.
\end{align}
%For $i=1,2$, % if $h_{k_i}$ is any solution of Equation \ref{eqhip} when $k=k_i$ (see [NR] for the existence of $h_{k_i}$ and other details), 
By Equation \ref{eqhip}, we obtain the inequalities
\begin{align*} 
2Gh_{A,k_1}''(s) + (1 + h_{A,k_1}'(s)^2)h_{A,k_1}'(s) G_s & < 0; \\
2Gh_{A,k_2}''(s) + (1 + h_{A,k_2}'(s)^2)h_{A,k_2}'(s) G_s & > 0,
\end{align*}
for any $A>0, i=1,2$.

The catenoid $C_{\mathbb{M}}^{A,k}$ separates $\mathbb{M} \times \mathbb{R}$ in two connected components. One of them contains $\mathbb{M} \times (T , + \infty)$, for some 
$T \in \mathbb{R}$, which we call the \textit{inner region} of $C_{\mathbb{M}}^{A,k}$. The other component is the \textit{outer region} of the catenoid.

We say that the mean curvature vector field $\overrightarrow{H}_{A,k}$ of $C_{\mathbb{M}}^{A,k}$ points \textit{inwards} (resp. \textit{outwards}) when it is nonzero 
everywhere and it points to the inner region (resp. to the outer region). With the above reasoning, we conclude the following result:

\begin{proposition} \label{catenoid} For a Hadamard surface $\mathbb{M}$, suppose that the inequalities $-k_1^2 < K_{\mathbb{M}}  < -k_2^2$ hold. Then, for any positive $A$, the 
vector field $\overrightarrow{H}_{A,k_1}$ points outwards, while $\overrightarrow{H}_{A,k_2}$ points inwards.
\end{proposition}

\textbf{Remark.} Concerning the variation of Proposition 2 of [GL], we need to assure that the inequalities in (\ref{lozano}) are strict, which is not done in the reference. 
Indeed, if $G^i(s , \theta) : = sinh^2(k_i s)$, for $i=1,2$, it is true that the functions $f_\theta(s) = \frac{G^1_s(s , \theta)}{2G^1(s , \theta)}$ and 
$g_\theta(s) = \frac{G_s(s , \theta)}{2G(s , \theta)}$ satisfy the equations
\begin{center}
 $f_{\theta}' + f_{\theta}^2 = k_1^2 > \frac{\displaystyle (- K_{\mathbb{M}}(\cdotp ,\theta) + k_1^2)}{\displaystyle 2}$; 
 $g_{\theta}' + g_{\theta}^2 = - K_{\mathbb{M}}(\cdotp ,\theta) < \frac{\displaystyle (-K_{\mathbb{M}}(\cdotp,\theta) + k_1^2)}{\displaystyle 2}$.
\end{center}
It is clear that $f_{\theta}$ and $g_{\theta}$ satisfy the conditions of Corollary 2.2 of [PRS] (see [GL] for details about $f_{\theta}$ and $g_{\theta}$). Then, 
defining
\begin{center}
 $\phi_{\theta}(s) = s\int_{0}^s (f_{\theta}(t) - t^{-1}) dt$;
 
 $\psi_{\theta}(s) = s\int_{0}^s (g_{\theta}(t) - t^{-1}) dt$,
\end{center}
we can apply the ideas of Lemma 2.1 of [PRS]. Explicitly,
\begin{align*}
& (\phi'_{\theta} \psi_{\theta} - \phi_{\theta} \psi'_{\theta})'(s) \geq (k_1^2 + K_{\mathbb{M}}(s , \theta))\phi_{\theta}(s) \psi_{\theta}(s) \leftrightarrow \\
& \frac{\displaystyle G^1_s(s , \theta)}{\displaystyle G^1(s , \theta)} - \frac{\displaystyle G_s(s , \theta)}{\displaystyle G(s , \theta)} \geq \frac{\displaystyle 2 \int_{0}^s (k_1^2 + K_{\mathbb{M}}(x , \theta))\phi_{\theta}(x) \psi_{\theta}(x) dx}{\phi_{\theta}\psi_{\theta}(s)},
\end{align*}
then one of the strict inequalities in (\ref{lozano}) was proved. The other one can be proved in a similar procedure. 
\begin{comment}Moreover, by continuity of the mean curvature vector field,
we obtain that the mean curvature of the rotational catenoids is nonzero also in the circle of vertical symmetry of the catenoid. 
\end{comment}

\subsubsection{Height bounds of minimal annuli}
We prove here the following proposition.

\begin{proposition}\label{cylinder} If $\mathbb{M}$ is a Cartan-Hadamard manifold and if $B_R(p)$ is a compact subset of $\mathbb{M}$, there exists  $h_0 > 0$ depending on 
$p$ and $R$ such that, for any two Jordan curves $C_1, C_2 \subset B_R(p)$ and $h' > h$, there is no minimal annulus in $\mathbb{M} \times \mathbb{R}$ whose boundary 
is given by $(C_1 \times \{ 0 \}) \cup (C_1 \times \{ h' \})$.
\end{proposition}
\begin{proof} Suppose, by contradiction, that there is a sequence $\{ \Sigma_n \}_{n \in \mathbb{N}}$ of minimal annuli such that 
$\partial \Sigma_n \subset \mathbb{M} \times \{ - h_n , h_n \}$, where $(h_n)_{n \in \mathbb{N}}$ is an increasing sequence of positive numbers which goes to $+\infty$. 
By [MY], there is a minimal stable annuli $S_n$ whose boundary is $\partial B_R(p) \times \{ -h_n , h_n \}$ that minimizes area among the annuli contained in the 
unbounded component of $(\mathbb{M} \times [-h_n , h_n]) \backslash \Sigma_n$. We then have area and curvature estimates for the sequence $(S_n)_n$ in compact sets, then, by 
a diagonal argument, we have that a subsequence of $(S_n)_n$ converges to a cylindrically bounded minimal annuli $S$. Since all the $S_n$ are stable, the surface $S$ 
also is. By Theorem 3 of [S], the second fundamental form of $S$ is bounded. 

%We then prove that the second fundamental form of $S$ is bounded. In fact, if $(p_n)_{n \in \mathbb{N}}$ is a sequence of points of $S$ such that 
%$|A(p_n)| \rightarrow +\infty$, we therefore have that, if $T_{p_n}$ is the translation in $\mathbb{R}^3$ by $-p_n$, the sequence $|A(p_n)|T_{p_n}S)$ converges (up to a 
%subsequence) to a complete oriented stable minimal surface $S'$ in $\mathbb{R}^3$, so $S'$ is a plane. On the other hand, $|A_{S'}(0)| = 1$, contradiction, 
%so the curvature is bounded.

Obviously, $S \subset B_R(p) \times \mathbb{R}$, and let $R'$ the smallest number such that $S \subset B_{R'}(p) \times \mathbb{R}$ (by the maximum principle, this number 
exists). By the choice of $R'$, we can choose a sequence $(s_n = (q_n, t_n))_{n \in \mathbb{N}}$ of points of $S$, $q_n \in \mathbb{M}, t_n \in \mathbb{R}$ such that $(q_n)_n$
converges to a point $q$ in $\partial B_R(p)$. We then consider, for each $n$, the surface $S^n$, a vertical translation of $S$ such that $\bar{s}_n := (q_n , 0) \in S^n$. 
The points $\bar{s}_n$ have $\delta$-neighborhoods on $S^n$ that are graphs of functions $F_n$ over the $\delta$-disc in $T_{\bar{s}_n} S^n$ such that the set $||F_n||_{C^2}$
is uniformly bounded. Therefore, up to a subsequence, the sequence $(T_{\bar{s}_n} S^n)$ converges to a vertical plane $P$ in $T_{(q,0)} (\mathbb{M} \times \mathbb{R})$, otherwise 
$S$ would not be contained in $B_R(p) \times \mathbb{R}$, and the sequence of graphs of $(F_n)_n$ converges to a minimal graph over a $\delta$-disc which intersects 
$\partial B_R(p) \times \mathbb{R}$ tangentially, which is impossible.  
% Choosing a positive number $l$ such that the slices $\mathbb{M} \times \{ \pm l\}$ intersect $S^n$ transversely, for all $n$, we take the sequence of the connected components of 
% $S^n \cap \mathbb{M} \times [-l , l]$ containing $(q_n , 0)$, say, $(S^n_1)_{n \in \mathbb{N}}$. As in the first paragraph, there is a minimal stable annuli $S^n_2$ 
% whose boundary is $\partial B_R(p) \times \{ -l , l \}$ that minimizes area among the annuli contained in the 
% unbounded component of $(\mathbb{M} \times [-l , l]) \backslash S^n_1$. We then have area and curvature estimates away from the boundary, then it is possible to obtain
% a convergent subsequence to a limit $S_*$ such that $S_* \cap (B_{R'}(p) \times \mathbb{R})$ in non-empty, contradicting the maximum principle.
\end{proof}

\section{References}
[CR] P. Collin, H. Rosenberg, Construction of harmonic diffeomorphisms and minimal graphs. Ann. Math. 172(3), 1879–1906 (2010).
\newline
[CY] S. Cheng, S.-T. Yau, Differential equations on Riemannian manifolds and their geometric applications, Comm. Pure Appl. Math. 28, (1975), 333-354.
\newline
[F] Y. Fang, Lectures on Minimal Surfaces in $\mathbb{R}^3$,  Australian National University, Centre for Mathematics and its Applications, Proceedings of the Centre for 
Mathematics and its Applications, Australian National University; v. 35, ISBN 0 7315 2443 8 (1996).
\newline
[Fr] K. Frensel, Stable complete surfaces with constant mean curvature, Bol. Soc. Brasil. Mat. 27, (1996) 129–144.
\newline
[GL] J.A. Gálvez, V. Lozano, Existence of barriers for surfaces with prescribed curvatures in $\mathbb{M}^2 \times \mathbb{R}$. J. Differ. Equ. 255(7), 1828–1838 (2013).
\newline
[GR] J.A. Gálvez, H. Rosenberg, Minimal surfaces and harmonic diffeomorphisms from the complex plane onto a Hadamard surface. Amer. J. Math., 132, 1249–1273 (2010).
\newline
[H] L. Hauswirth, Generalized Riemann examples in three-dimensional manifolds, Pacific Journal of Math., 224, no. 1, (2006), 91-117.
\newline
[Hu] A. Huber: On subharmonic functions and differential geometry in the large. Comm. Math. Helv.32, (1957), 13–72.
\newline
[HK] D. Hoffman, H. Karcher, Complete embedded minimal surfaces of finite total curvature, in: “Geometry, V”, Encyclopaedia Math. Sci. 90, Springer, Berlin, 5–93, 267–272 (1997).
\newline
[HNST] L. Hauswirth, B. Nelli, R. Sa Earp, E. Toubiana, Minimal ends in $\mathbb{H}^2 \times \mathbb{R}$ with finite total curvature and a Schoen type theorem, Advances in 
Mathematics, 274 (2015), 199–240.
\newline
[HR] L. Hauswirth, H. Rosenberg, Minimal surfaces of finite total curvature in $\mathbb{H} \times \mathbb{R}$. Mat. Contemp. 31, (2006), 65–80. 
\newline
[HSET] L. Hauswirth, R. Sa Earp, E. Toubiana, Associate and conjugate minimal immersions in $\mathbb{H}^2 \times \mathbb{R}$, Tohoku Math. J. 60(2), (2008), 267–286.
\newline
[JM]  L. Jorge, W. Meeks III, The topology of complete minimal surfaces of finite total Gaussian curvature, Topology 22 (1983), 203-221.  
\newline
[LTW] P. Li, L. Tam, J. Wang, Harmonic diffeomorphisms between Hadamard manifolds, Tran. AMS., 347, (1995), 3645–3658.
\newline
[MR] F. Morabito, M. M. Rodriguez, Saddle towers and minimal k–noids in $\mathbb{H}^2 \times \mathbb{R}$, J. Inst. Math. Jussieu 11 (2012), 333–349. MR2905307
\newline
[MY] W. H. Meeks, III and S. T. Yau. The existence of embedded minimal surfaces and the problem of uniqueness. Math. Z., 179(2):151–168, 1982.
\newline
% [NR] B. Nelli, H. Rosenberg, Minimal surfaces in $\mathbb{H}^2 \times \mathbb{R}$, Bull. Braz. Math. Soc. (N.S.) 33 (2002), no. 2, 263–292. MR1940353 (2004d:53014).
% \newline
[O] R. Osserman, A survey of minimal surfaces, 2nd edition, Dover Publications, New York, 1986. 
\newline
[P] R. Ponte, Minimal annuli in $\mathbb{M}^2 \times \mathbb{R}$, in preparation.
\newline
[Py] J. Pyo, New complete embedded minimal surfaces in $\mathbb{M}^2 \times \mathbb{R}$, Ann. Global Anal. Geom. 40, (2011), 167–176. MR2811623
\newline
[PR] J. Pyo, M.M. Rodriguez, Simply-connected minimal surfaces with finite total curvature in $\mathbb{H}^2 \times \mathbb{R}$, Int. Math. Res. Notices, 2014, 2944-2954 (2014).
\newline
[PRS] S. Pigola, M. Rigoli, A. G. Setti, Vanishing and finiteness results in geometric analysis: a generalization of the Bochner technique, Progress in Mathematics 266, 
Birkhäuser, Basel, 2008.  MR 2009m:58001  Zbl 1150.53001
\newline
[S] R. Schoen. Estimates for stable minimal surfaces in three-dimensional manifolds. In Seminar on minimal submanifolds, volume 103 of Ann. of Math. Stud., pages 111–126.
Princeton Univ. Press, 1983. MR795231 (86j:53094)
\newline
[SY] R. Schoen, S. T. Yau, Lectures on harmonic maps, Conf. Proc. Lecture Notes Geom. and Topology, II. International Press, Cambridge, Mass., (1997).
\newline
[St] K. Strebel, Quadratic differentials, Ergebnisse der Mathematik und ihrer Grenzgebiete (3) [Results in Mathematics and Related Areas (3)], vol. 5, Springer-Verlag, 1984.
\newline
[W] B. White. Complete surfaces of finite total curvature, J. Diff. Geom., 26:315-326, 1987.

\Addresses

\end{document}